\documentclass[reqno]{amsart}
\usepackage{hyperref,color}
\usepackage{amsmath}

\begin{document}
\title%
[On the essential norms of Toeplitz operators]
{On the essential norms of Toeplitz operators\\
on abstract Hardy spaces built upon
Banach function spaces}

\author[Oleksiy Karlovych, Eugene Shargorodsky\hfil \hfilneg] 
{Oleksiy Karlovych, Eugene Shargorodsky}  

\address{Oleksiy Karlovych \newline
Centro de Matem\'atica e Aplica\c{c}\~oes,
Departamento de Matem\'a\-tica,
Faculdade de Ci\^en\-cias e Tecnologia,
Universidade Nova de Lisboa,
Quinta da Torre,
2829--516 Caparica, Portugal}
\email{oyk@fct.unl.pt}

\address{Eugene Shargorodsky \newline
Department of Mathematics,
King's College London,
Strand, London WC2R 2LS,
United Kingdom}
\email{eugene.shargorodsky@kcl.ac.uk}
\begin{abstract}
Let $X$ be a Banach function space over the unit circle
such that the Riesz projection $P$ is bounded on $X$ and let $H[X]$ be the 
abstract Hardy space built upon $X$. We show that the essential norm of the 
Toeplitz operator $T(a):H[X]\to H[X]$ coincides with $\|a\|_{L^\infty}$ for 
every $a\in C+H^\infty$ if and only if the essential norm of the backward
shift operator $T(\mathbf{e}_{-1}):H[X]\to H[X]$ is equal to one, where
$\mathbf{e}_{-1}(z)=z^{-1}$. This result extends an observation by
B\"ottcher, Krupnik, and Silbermann for the case of classical Hardy spaces.
\end{abstract}
\dedicatory{To the memory of Nikolai Vasilevski}

\maketitle \numberwithin{equation}{section}
\newtheorem{theorem}{Theorem}[section]
\newtheorem{corollary}[theorem]{Corollary}
\newtheorem{lemma}[theorem]{Lemma}
\newtheorem{remark}[theorem]{Remark}
\newtheorem{problem}[theorem]{Problem}
\newtheorem{example}[theorem]{Example}
\newtheorem{definition}[theorem]{Definition}
\allowdisplaybreaks
\section{Introduction and the main result}
For a Banach space $\mathcal{X}$, let $\mathcal{B}(\mathcal{X})$ denote the 
Banach algebra of bounded linear operators on $\mathcal{X}$ and let 
$\mathcal{K}(\mathcal{X})$ be the closed two-sided ideal of 
$\mathcal{B}(\mathcal{X})$ consisting of all compact linear operators
on $\mathcal{X}$. The norm of an operator $A\in\mathcal{B}(\mathcal{X})$ 
is denoted by $\|A\|_{\mathcal{B}(\mathcal{X})}$. The essential norm of 
$A \in \mathcal{B}(\mathcal{X})$ 
is defined as follows:
\[
\|A\|_{\mathcal{B}(\mathcal{X}),\mathrm{e}} 
:= 
\inf\{\|A - K\|_{\mathcal{B}(\mathcal{X})}\ : \  K \in \mathcal{K}(\mathcal{X})\}.
\]

For a function $f \in L^1$ on the unit circle 
$\mathbb{T} :=\left\{z \in \mathbb{C} : \ |z| = 1\right\}$
equipped with the Lebesgue measure $m$ normalised so that $m(\mathbb{T})=1$, 
let
\[
\widehat{f}(n) 
= 
\frac{1}{2\pi} \int_{-\pi}^\pi 
f\left(e^{i\theta}\right) e^{-i n\theta}\, d\theta , 
\quad 
n \in \mathbb{Z}
\]
be the Fourier coefficients of $f$. Let $X$ be a Banach function space on 
the unit circle $\mathbb{T}$. We postpone the definition of this notion until 
Section~\ref{subsec:BFS}. Here we only mention that the class of 
Banach function spaces is very reach, it includes all Lebesgue spaces 
$L^p$, $1\le p\le\infty$, Orlicz spaces $L^\Phi$ 
(see, e.g., \cite[Ch.~4, Section~8]{BS88}), 
and Lorentz spaces $L^{p,q}$, $1<p<\infty$, $1\le q\le\infty$ 
(see, e.g., \cite[Ch.~4, Section~4]{BS88}). Moreover, all mentioned
above spaces are rearrangement-invariant (see Section~\ref{subsec:riBFS}
for their definition).

Let
\[
H[X] := 
\left\{
g\in X\ :\ \widehat{g}(n)=0\quad\mbox{for all}\quad n<0
\right\} 
\]
denote the abstract Hardy space built upon the space $X$. In the case $X = L^p$, 
where $1\le p\le\infty$, we will use the standard notation $H^p := H[L^p]$.
Consider the operators $S$ and $P$, defined for a function 
$f\in L^1$ and an a.e. point $t\in\mathbb{T}$ by
\[
(Sf)(t):=\frac{1}{\pi i}\,\mbox{p.v.}\int_\mathbb{T}
\frac{f(\tau)}{\tau-t}\,d\tau,
\quad
(Pf)(t):=\frac{1}{2}(f(t)+(Sf)(t)),
\]
respectively, where the integral is understood in the Cauchy principal 
value sense. The operator $S$ is called the Cauchy singular 
integral operator and the operator $P$ is called the Riesz projection.
Assume that the Riesz projection is bounded on $X$.
For $a\in L^\infty$, the Toeplitz operator with symbol $a$ is defined by
\[
T(a)f=P(af),
\quad 
f\in H[X].
\]
It is clear that $T(a)\in\mathcal{B}(H[X])$ and 
\begin{equation}\label{eq:Toeplitz-bound}
\|T(a)\|_{\mathcal{B}(H[X]),{\rm e}}
\le 
\|T(a)\|_{\mathcal{B}(H[X])}
\le 
\|P\|_{\mathcal{B}(X)}\|a\|_{L^\infty}.
\end{equation}
Let $C$ denote the Banach space of all complex-valued continuous functions
on $\mathbb{T}$ with the supremum norm and let 
\[
C+H^\infty:=\{f\in L^\infty\ :\ f=g+h,\ g\in C,\ h\in H^\infty\}.
\]
In 1967, Sarason observed that $C+H^\infty$ is a closed subalgebra
of $L^\infty$ (see, e.g., \cite[Ch.~IX, Theorem~2.2]{G07} for the proof 
of this fact).

Let $1<p<\infty$ and $a\in L^\infty$. It follows from 
\cite[Theorem~2.30]{BS06} and \eqref{eq:Toeplitz-bound} that
\[
\|a\|_{L^\infty}
\le 
\|T(a)\|_{\mathcal{B}(H^p),\mathrm{e}}
\le 
\|T(a)\|_{\mathcal{B}(H^p)}
\le
\|P\|_{\mathcal{B}(L^p)}\|a\|_{L^\infty}.
\]
I.~Gohberg and N.~Krupnik \cite[Theorem~6]{GK68} proved that 
$\|P\|_{\mathcal{B}(L^p),\mathrm{e}}\ge 1/\sin(\pi/p)$ and conjectured that
$\|P\|_{\mathcal{B}(L^p)}=1/\sin(\pi/p)$. This conjecture was confirmed
by B.~Hollenbeck and I.~Verbitsky in \cite{HV00}. Thus
\begin{equation}\label{eq:ess-norm-Toeplitz}
\|a\|_{L^\infty} 
\le
\|T(a)\|_{\mathcal{B}(H^p),\mathrm{e}}
\le 
1/\sin(\pi/p)
\|a\|_{L^\infty},
\quad
a\in L^\infty.
\end{equation}
A.~B\"ottcher, N.~Krupnik, and B.~Silbermann \cite[Section~7.6]{BKS88} 
asked whether the essential norm of Toeplitz operators $T(a)$ with 
$a\in C$ acting on the Hardy spaces $H^p$ is independent of $p\in(1,\infty)$. 
The second author answered this question in the negative \cite{S21}. 
More precisely, it was shown that 
\begin{equation}\label{eq:Shardgorodsy-first-result}
\|T(a)\|_{\mathcal{B}(H^p),\mathrm{e}}=\|a\|_{L^\infty}
\quad\mbox{for all}\quad a\in C
\quad\mbox{if and only if}\quad 
p=2.
\end{equation}
Nevertheless, the following estimates
for $\|T(a)\|_{\mathcal{B}(H^p),\mathrm{e}}$ were obtained for $1<p<\infty$
and $a\in C+H^\infty$:
\[
\|a\|_{L^\infty} 
\le 
\|T(a)\|_{\mathcal{B}(H^p),\mathrm{e}}  
\le 
\min\left\{2^{|1-2/p|},1/\sin(\pi/p)\right\} \|a\|_{L^\infty}
\]
(see \eqref{eq:ess-norm-Toeplitz} and \cite[Theorem~4.1]{S21}).

We will use the following notation:
\[
\mathbf{e}_m(z) := z^m, 
\quad z \in \mathbb{C}, 
\quad m \in \mathbb{Z}.
\]

The following result extends \eqref{eq:Shardgorodsy-first-result} to the class
of rearrangement-invariant Banach function spaces.
\begin{theorem}\label{th:known}
Let $X$ be a rearrangement-invariant Banach function space such that the
Riesz projection $P$ is bounded on $X$. Then the following statements
are equivalent:

\begin{enumerate}
\item[(a)] the equality
\begin{equation}\label{eq:known}
\|T(a)\|_{\mathcal{B}(H[X]),\mathrm{e}}=\|a\|_{L^\infty}
\end{equation}
holds for every Toeplitz operator $T(a):H[X]\to H[X]$ with $a\in L^\infty$;

\item[(b)] equality \eqref{eq:known}
holds for every Toeplitz operator $T(a):H[X]\to H[X]$ with $a\in C+H^\infty$;

\item[(c)] 
$\|T(\mathbf{e}_{-1})\|_{\mathcal{B}(H[X]),\mathrm{e}}=1$;

\item[(d)]
$\|T(\mathbf{e}_{-1})\|_{\mathcal{B}(H[X])}=1$;

\item[(e)]
$\|P\|_{\mathcal{B}(X)}=1$;

\item[(f)]
$X=L^2$ and there exists $C\in(0,\infty)$ such that
\[
\|g\|_X=C\|g\|_{L^2}\quad\mbox{for all}\quad g\in X.
\]
\end{enumerate}
\end{theorem}
The implication (f) $\Longrightarrow$ (a) follows from inequalities 
\eqref{eq:ess-norm-Toeplitz}, which become equalities for $p=2$. 
The implications (a) $\Longrightarrow$ (b) $\Longrightarrow$ (c) are 
trivial. The equivalences 
(d) $\Longleftrightarrow$ (e) $\Longleftrightarrow$ (f)
were proved in \cite[Theorems~1.1--1.2]{KS23-JFA} for arbitrary 
(not necessarily rearrangement-invariant) Banach function spaces $X$. The 
equality 
\[
\|T(\mathbf{e}_{-1})\|_{\mathcal{B}(H[X])}=
\|T(\mathbf{e}_{-1})\|_{\mathcal{B}(H[X]),\mathrm{e}}
\]
was proved in \cite[Theorem~1.2]{KS22-PAMS} for rearrangement-invariant 
Banach function spaces $X$, which gives the equivalence 
(c) $\Longleftrightarrow$ (d) and completes the proof of 
Theorem~\ref{th:known}.

B\"ottcher, Krupnik, and Silbermann \cite[p.~472]{BKS88} provided
an argument allowing to show directly  that (c) $\Longrightarrow$ (b) in 
the case of classical Hardy spaces $H^p$, $1<p<\infty$. The aim of this paper
is to show that their reasoning can be extended to the case of arbitrary
Banach function spaces (not necessarily rearrangement-invariant)
on which the Riesz projection $P$ is bounded. Our main result is the 
following. 
\begin{theorem}[Main result]
\label{th:BKS} 
Let $X$ be a Banach function space on which the Riesz projection is bounded.
Then the following statements are equivalent:
\begin{enumerate}
\item[(i)]
the equality
\[
\|T(\mathbf{e}_{-1})\|_{\mathcal{B}(H[X]),\mathrm{e}} = 1 
\]
holds for the backward shift operator $T(\mathbf{e}_{-1}) : H[X] \to H[X]$;
 
\item[(ii)]
the equality
\[
\|T(a)\|_{\mathcal{B}(H[X]),\mathrm{e}}  = \|a\|_{L^\infty} 
\]
holds for every Toeplitz operator $T(a):H[X]\to H[X]$ with
$a \in C + H^\infty$.
\end{enumerate}
\end{theorem}

The paper is organised as follows. In Section~\ref{sec:preliminaries},
we recall the definition of the class of Banach function spaces and of its
distinguished subclasss of rearrangement-invariant Banach function spaces.
In Section~\ref{sec:auxiliary}, we prove that the Toeplitz operators
$T(\mathbf{e}_{-n}h)$ with $n\in\mathbb{Z}_+:=\{0,1,2,\dots\}$ and 
$h\in H^\infty$ are bounded on $H[X]$. Further, we 
show that $T(\mathbf{e}_{-1})T(\mathbf{e}_{-n}h)=T(\mathbf{e}_{-n-1}h)$
for $n\in\mathbb{Z}_+$ and $h\in H^\infty$ on the space $H[X]$.
Although our main results have been obtained under the assumption that 
$P$ is bounded on $X$, we do not make this assumption in Section~\ref{sec:auxiliary}
as we believe that this more general case is of an independent interest.
By using the results of Section~\ref{sec:auxiliary}, we prove 
Theorem~\ref{th:BKS} in Section~\ref{sec:proof-of-BKS}.
\section{Preliminaries}\label{sec:preliminaries}
\subsection{Banach function spaces}\label{subsec:BFS}
Let $\mathcal{M}$ be the set of all measurable complex-valued functions on 
$\mathbb{T}$ equipped with the normalized Lebesgue measure $m$ and let 
$\mathcal{M}^+$ be the subset of functions in $\mathcal{M}$ whose values lie 
in $[0,\infty]$. Following \cite[Ch.~1, Definition~1.1]{BS88}, a mapping 
$\rho: \mathcal{M}^+\to [0,\infty]$ is called a Banach function norm if, 
for all functions $f,g, f_n\in \mathcal{M}^+$ with $n\in\mathbb{N}$, and for 
all constants $a\ge 0$, the following  properties hold:
\begin{eqnarray*}
{\rm (A1)} & &
\rho(f)=0  \Leftrightarrow  f=0\ \mbox{a.e.},
\
\rho(af)=a\rho(f),
\
\rho(f+g) \le \rho(f)+\rho(g),\\
{\rm (A2)} & &0\le g \le f \ \mbox{a.e.} \ \Rightarrow \ 
\rho(g) \le \rho(f)
\quad\mbox{(the lattice property)},\\
{\rm (A3)} & &0\le f_n \uparrow f \ \mbox{a.e.} \ \Rightarrow \
       \rho(f_n) \uparrow \rho(f)\quad\mbox{(the Fatou property)},\\
{\rm (A4)} & & \rho(1) <\infty,\\
{\rm (A5)} & &\int_{\mathbb{T}} f(t)\,dm(t) \le C\rho(f)
\end{eqnarray*}
with {a constant} $C \in (0,\infty)$ that may depend on $\rho$,  but is 
independent of $f$. When functions differing only on a set of measure zero 
are identified, the set $X$ of all functions $f\in \mathcal{M}$ for which 
$\rho(|f|)<\infty$ is called a Banach function space. For each $f\in X$, 
the norm of $f$ is defined by $\|f\|_X :=\rho(|f|)$. The set $X$ equipped 
with the natural linear space operations and this norm becomes a Banach 
space (see \cite[Ch.~1, Theorems~1.4 and~1.6]{BS88}). 
\subsection{Rearrangement-invariant Banach function spaces}\label{subsec:riBFS}
Let $\mathcal{M}_0$ (resp. $\mathcal{M}_0^+$)
denote the set of all a.e. finite functions in $\mathcal{M}$
(resp. in $\mathcal{M}^+$).
Following \cite[Chap.~2, Definitions~1.1 and~1.2]{BS88}, the 
distribution function $m_f$ of a function $f\in \mathcal{M}_0$ 
is given by
\[
m_f(\lambda) := 
m\left\{t \in\mathbb{T}\ :\  |f(t)| > \lambda\right\},
\quad\lambda\ge 0.
\]
Two functions $f,g\in \mathcal{M}_0$ are said to be 
equimeasurable if $m_f(\lambda)=m_g(\lambda)$ for all $\lambda\ge 0$. 
A Banach function norm $\rho:\mathcal{M}\to[0,\infty]$ is said 
to be rearrangement-invariant if $\rho(f)=\rho(g)$ for every pair of
equimeasurable functions $f,g\in\mathcal{M}_0^+$. 
In that case, the Banach function space $X$ generated by $\rho$ is 
said to be a rearrangement-invariant Banach function space (see 
\cite[Ch.~2, Definition~4.1]{BS88}). 
\section{Auxiliary results}\label{sec:auxiliary}
\subsection{Operator $P_n$}
For $n\in\mathbb{N}$ and $f\in L^1$, put
\[
P_n f := \sum_{k = 0}^{n-1} \widehat{f}(k) \mathbf{e}_k\in H^1.
\]
\begin{lemma}\label{le:boundedness-Pn}
For every $n\in\mathbb{N}$, the operator $P_n : L^1 \to H^\infty$ is bounded
and 
\[
\|P_n\|_{\mathcal{B}(L^1,H^\infty)}\le n.
\]
\end{lemma}
\begin{proof}
For every $f \in L^1$, on has
\begin{align*}
\|P_nf\|_{L^\infty} 
&= 
\left\|\sum_{k = 0}^{n-1} \widehat{f}(k) \mathbf{e}_k\right\|_{L^\infty} 
\le 
\sum_{k = 0}^{n-1} \left|\widehat{f}(k)\right| \|\mathbf{e}_k\|_{L^\infty}
\\
&= 
\sum_{k = 0}^{n-1} \left|\widehat{f}(k)\right| 
\le 
\sum_{k = 0}^{n-1} \|f\|_{L^1} = n   \|f\|_{L^1}.
\end{align*}
So, $P_n\in\mathcal{B}(L^1, H^\infty)$ and 
$\|P_n\|_{\mathcal{B}(L^1, H^\infty)}\le n$.
\end{proof}
\begin{corollary}\label{cor:boundedness-Pn}
Let $X$ be a Banach function space. For every $n\in\mathbb{N}$,
the operator $P_n : X \to H[X]$ is bounded.
\end{corollary}
\begin{proof}
Axioms (A4) and (A5) imply the existence of a constant $C>0$ such that
\[
\|P_n\|_{\mathcal{B}(X,H[X])}  
\le C  
\|P_n\|_{\mathcal{B}(L^1, H^\infty)}
\le C n,
\]
which completes the proof.
\end{proof}
\subsection{Boundendess of a special Toeplitz operator}
We will need the following auxiliary lemma
\begin{lemma}[{\cite[Lemma~3.1]{K21}}]
\label{le:Pf=g}
Let $f \in L^1$. Suppose there exists $g \in H^1$ such that 
$\widehat{f}(n) = \widehat{g}(n)$ for all $n \ge 0$. Then $Pf = g$.
\end{lemma}
As a consequence of the results of the previous subsection
and Lemma~\ref{le:Pf=g}, we will
show that special Toeplitz operators with symbols of the form 
$\mathbf{e}_{-n}h$, where $n\in\mathbb{N}$ and $h\in H^\infty$, are bounded
on abstract Hardy spaces $H[X]$ built upon Banach function spaces $X$ 
even without the assumption that the Riesz projection $P$ is bounded on $X$.
\begin{lemma}\label{le:boundedness-T-special}
Let $X$ be a Banach function space. If $n \in \mathbb{Z}_+$ and $h \in H^\infty$, 
then the Toeplitz operator $T(\mathbf{e}_{-n} h) : H[X] \to H[X]$ is bounded.
\end{lemma}
\begin{proof}
Let $f\in H[X]\subset H^1$. Since $h\in H^\infty$, it follows from
\cite[Section~3.3.1, properties (a), (g)]{N19} that $hf\in H^1$. 
In the case $n=0$, Lemma~\ref{le:Pf=g} implies that
\begin{equation}\label{eq:boundedness-T-special-1}
T(h)f=P(hf)=hf
\end{equation}
and
\[
\|T(h)f\|_{H[X]}=\|hf\|_{H[X]}\le \|h\|_{L^\infty}\|f\|_{H[X]},
\]
whence
\begin{equation}\label{eq:boundedness-T-special-2}
\|T(h)\|_{\mathcal{B}(H[X])}\le \|h\|_{L^\infty}.
\end{equation}

If $n\in\mathbb{N}$, then
\begin{align*}
T(\mathbf{e}_{-n} h) f 
= 
P(\mathbf{e}_{-n} h f) =  P(\mathbf{e}_{-n} P_n(h f)) + P(\mathbf{e}_{-n} (I - P_n)(h f)) .
\end{align*}
It follows from the definition of $P_n$ that
\begin{align}\label{eq:boundedness-T-special-add}
& (\mathbf{e}_{-n} P_n(h f))\widehat{\hspace{2mm}}(m) = 0, 
\quad m\in\mathbb{Z}_+;
\\
& (\mathbf{e}_{-n} (I-P_n)(h f))\widehat{\hspace{2mm}}(m) = 0, 
\quad 
m\in\mathbb{Z}\setminus\mathbb{Z}_+. \nonumber
\end{align}
Hence 
\[
P(\mathbf{e}_{-n} P_n(h f)) = 0,
\quad
P(\mathbf{e}_{-n} (I - P_n)(h f)) = \mathbf{e}_{-n} (I - P_n)(h f)
\]
(see Lemma~\ref{le:Pf=g}). So,
\begin{equation}\label{eq:boundendess-T-special-6}
T(\mathbf{e}_{-n} h) f 
= 
P(\mathbf{e}_{-n} h f) = \mathbf{e}_{-n} (I - P_n)(hf) .
\end{equation} 
Hence, taking into account Corollary~\ref{cor:boundedness-Pn}, we obtain
\begin{align*}
\|T(\mathbf{e}_{-n} h) f\|_{H[X]} 
& =
\|T(\mathbf{e}_{-n} h) f\|_X 
=
\|\mathbf{e}_{-n} (I - P_n)(hf)\|_X \\
& =
\|(I - P_n)(hf)\|_X 
\le 
\left(1 + \|P_n\|_{\mathcal{B}(X,H[X])}\right) \|hf\|_X 
\\
& = 
\left(1 + \|P_n\|_{\mathcal{B}(X,H[X])}\right) 
\|h\|_{L^\infty} \|f\|_{H[X]} .
\end{align*}
So, $T(\mathbf{e}_{-n} h) \in\mathcal{B}(H[X])$ and 
\[
\|T(\mathbf{e}_{-n} h)\|_{\mathcal{B}(H[X])} 
\le 
\left(1 + \|P_n\|_{\mathcal{B}(X,H[X])}\right) \|h\|_{L^\infty},
\]
which completes the proof.
\end{proof}
The above lemma can be complemented by the following 
(cf. \cite[Proposition~2.14]{BS06}).
\begin{lemma}\label{le:factorization}
Let $X$ be a Banach function space. If $n\in\mathbb{Z}_+$ and $h\in H^\infty$,
then
\[
T(\mathbf{e}_{-1})T(\mathbf{e}_{-n}h)=T(\mathbf{e}_{-n-1}h)
\]
on the space $H[X]$.
\end{lemma}
\begin{proof}
It follows from Lemma~\ref{le:boundedness-T-special} that
the Toeplitz operators $T(\mathbf{e}_{-1})$, $T(\mathbf{e}_{-n}h)$ and
$T(\mathbf{e}_{-n-1}h)$ are bounded on the space $H[X]$. Let $f\in H[X]$.
If $n=0$, then it follows from \eqref{eq:boundedness-T-special-1}
that
\[
T(\mathbf{e}_{-1})T(h)f
=
T(\mathbf{e}_{-1})(hf)
=
P(\mathbf{e}_{-1}hf)
=
T(\mathbf{e}_{-1}h)f.
\]
If $n\in\mathbb{N}$, then 
\eqref{eq:boundendess-T-special-6} and \eqref{eq:boundedness-T-special-add}
imply that
\begin{align*}
T(\mathbf{e}_{-1})T(\mathbf{e}_{-n}h)f
=&
\mathbf{e}_{-1}(I-P_1)(T(\mathbf{e}_{-n}h)f)
\\
=&
\mathbf{e}_{-1}\left[
T(\mathbf{e}_{-n}h)f-(T(\mathbf{e}_{-n}h)f)\widehat{\hspace{2mm}}(0)
\right]
\\
=&
\mathbf{e}_{-1} \left[
\mathbf{e}_{-n}(I-P_n)(hf)
-
(\mathbf{e}_{-n}(I-P_n)(hf))\widehat{\hspace{2mm}}(0)
\right]
\\
=&
\mathbf{e}_{-n-1}(I-P_n)(hf)
-
\mathbf{e}_{-1}(\mathbf{e}_{-n}hf)\widehat{\hspace{2mm}}(0)
\\
=&
\mathbf{e}_{-n-1}(I-P_{n+1})(hf)
\\
&+
\mathbf{e}_{-n-1}(P_{n+1}-P_n)(hf)
-
\mathbf{e}_{-1}(\mathbf{e}_{-n}hf)\widehat{\hspace{2mm}}(0)
\\
=&
T(\mathbf{e}_{-n-1}h)f
+
\mathbf{e}_{-n-1}\widehat{hf}(n)\mathbf{e}_n
-
\mathbf{e}_{-1}\widehat{hf}(n)
\\
=&
T(\mathbf{e}_{-n-1}h)f,
\end{align*}
which completes the proof.
\end{proof}
\section{Proof of the main result}\label{sec:proof-of-BKS}
\subsection{Extending an observation by B\"ottcher, Krupnik, and Silbermann}
We start with the following auxiliary result, containing the essence 
of the argument in \cite[p.~472]{BKS88}, in which 
we do not assume the boundedness of the Riesz projection on a
Banach function space $X$.
\begin{lemma}\label{le:BKS}
Let $X$ be a Banach function space. If
$\|T(\mathbf{e}_{-1})\|_{\mathcal{B}(H[X]),\mathrm{e}} = 1$,
then $\|T(a)\|_{\mathcal{B}(H[X]),\mathrm{e}}  \le \|a\|_{L^\infty}$
for every Toeplitz operator $T(a):H[X]\to H[X]$ with
$a \in \{\mathbf{e}_{-n}h : n\in\mathbb{N}, h\in H^\infty\}$.
\end{lemma}
\begin{proof}
Let $a=\mathbf{e}_{-n}h$ for some $n\in\mathbb{N}$ and $h\in H^\infty$.
By Lemma~\ref{le:factorization} and \eqref{eq:boundedness-T-special-2}, 
we have
\begin{align*}
\|T(a)\|_{\mathcal{B}(H[X]),\mathrm{e}} 
& = 
\|T(\mathbf{e}_{-n}h)\|_{\mathcal{B}(H[X]),\mathrm{e}} 
=  
\|(T(\mathbf{e}_{-1}))^nT(h)\|_{\mathcal{B}(H[X]),\mathrm{e}}
\\
& \le 
\|T(\mathbf{e}_{-1})\|_{\mathcal{B}(H[X]),\mathrm{e}}^n 
\|T(h)\|_{\mathcal{B}(H[X]),\mathrm{e}} 
=
\|T(h)\|_{\mathcal{B}(H[X]),\mathrm{e}} 
\\
& \le 
\|T(h)\|_{\mathcal{B}(H[X])} 
\le  
\|h\|_{L^\infty} 
= 
\|a\|_{L^\infty},
\end{align*}
which completes the proof.
\end{proof}
\subsection{Proof of Theorem~\ref{th:BKS}}
It is clear that (ii) implies (i). Suppose (i) holds and $a\in C+H^\infty$. Since 
the set $G:=\{\mathbf{e}_{-n}h : n\in\mathbb{N}, h\in H^\infty\}$ is
dense in $C + H^\infty$ (see, e.g., \cite[Ch. IX, Theorem 2.2]{G07}), 
there is a sequence $\{a_m\}$ of elements of $G$ such that
$\|a-a_m\|_{L^\infty}\to 0$ as $m\to\infty$. By Lemma~\ref{le:BKS},
$\|T(a_m)\|_{\mathcal{B}(H[X]),\mathrm{e}}  \le \|a_m\|_{L^\infty}$
for all $m\in \mathbb{N}$.
On the other hand, it follows from \cite[Theorem~5.2]{KS23-IEOT}
that $\|a_m\|_{L^\infty} \le \|T(a_m)\|_{\mathcal{B}(H[X]),\mathrm{e}}$
for all $m\in\mathbb{N}$. Thus 
$\|T(a_m)\|_{\mathcal{B}(H[X]),\mathrm{e}}  = \|a_m\|_{L^\infty}$
for all $m\in \mathbb{N}$. Since
\[
\left|\|T(a)\|_{\mathcal{B}(H[X]),\mathrm{e}}
-
\|T(a_m)\|_{\mathcal{B}(H[X]),\mathrm{e}}
\right|
\le
\|T(a-a_m)\|_{\mathcal{B}(H[X]),\mathrm{e}} 
\le 
\|P\|_{\mathcal{B}(X)} \|a-a_m\|_{L^\infty}
\]
and $\left|\,
\|a\|_{L^\infty}-\|a_m\|_{L^\infty}
\right|
\le 
\|a-a_m\|_{L^\infty}$ for all $m\in\mathbb{N}$ and $\|a-a_m\|_{L^\infty}\to 0$
as $m\to\infty$, we get
\[
\|T(a)\|_{\mathcal{B}(H[X]),\mathrm{e}}
=
\lim_{n\to\infty}\|T(a_m)\|_{\mathcal{B}(H[X]),\mathrm{e}}
=
\lim_{m\to\infty}\|a_m\|_{L^\infty}=\|a\|_{L^\infty},
\]
which completes the proof of (ii).
\qed
\section*{Acknowledgements}
This work is funded by national funds through the FCT - Funda\c{c}\~ao para a 
Ci\^encia e a Tecnologia, I.P., under the scope of the projects UIDB/00297/2020 
(\url{https://doi.org/10.54499/UIDB/00297/2020})
and UIDP/ 00297/2020 
(\url{https://doi.org/10.54499/UIDP/00297/2020})
(Center for Mathematics and Applications).
\bibliographystyle{abbrv}
\bibliography{OKES-BSMM} 
\end{document}